\documentclass{gtart}

\usepackage{psfrag, graphicx, subfigure, epsfig}

\usepackage{amsmath, amssymb, latexsym, euscript}

\usepackage{mathptmx, wrapfig}

\theoremstyle{plain}
\newtheorem{theorem}{Theorem}
\newtheorem{lemma}[theorem]{Lemma}
\newtheorem{proposition}[theorem]{Proposition}

\theoremstyle{definition}

\newtheorem*{definition*}{Definition}
\newtheorem{algorithm}[theorem]{Algorithm}

\theoremstyle{remark}
\newtheorem{remark}[theorem]{Remark}

\numberwithin{equation}{section}

\def\bkE{{\rm I\kern-.22em E}}
 
\def\bkH{{\rm I\kern-.22em H}}
 
\def\bkR{{\rm I\kern-.17em R}}
\def\RR{\bkR}

\def\tri{\mathcal{T}}

\def\R3{{\Bbb R}^3}

\DeclareMathOperator{\cone}{cone}

\begin{document}

\title{The Thurston norm via Normal Surfaces}
\author{Daryl Cooper and Stephan Tillmann}
\dedication{To Bill Thurston on the occasion of his sixtieth birthday}

\begin{abstract} 
Given a triangulation of a closed, oriented, irreducible, atoroidal 3--mani\-fold every oriented, incompressible surface may be isotoped into normal position relative to the triangulation. Such a normal oriented surface is then encoded by non-negative integer weights, 14 for each 3--simplex, that describe how many copies of each oriented normal disc type there are. The Euler characteristic and homology class are both linear functions of the weights. There is a convex polytope in the space of weights, defined by linear equations given by the combinatorics of the triangulation, whose image under the homology map is the unit ball, $\mathcal{B},$ of the Thurston norm.

Applications of this approach include (1) an algorithm to compute $\mathcal{B}$ and hence the Thurston norm of any homology class, (2) an explicit exponential bound on the number of vertices of $\mathcal{B}$ in terms of the number of simplices in the triangulation, (3) an algorithm to determine the fibred faces of $\mathcal{B}$ and hence an algorithm to decide whether a 3--manifold fibres over the circle.
\end{abstract}

\makeshorttitle

This work was inspired by a desire to understand the topological significance of the faces of the unit ball of the Thurston norm. The main result of this paper implies that the unit ball of the Thurston norm for a closed, orientable, irreducible, atoroidal three-manifold is the projection under a linear map of a certain polyhedron in transversely oriented normal surface space. This polyhedron can be computed from a triangulation using  linear algebra.  In order to make this paper accessible to a wide audience we do not assume that the reader is familiar with the theory of normal surfaces. To keep the paper short we have considered only the case that $M$ is closed and orientable and have introduced the bare minimum of the theory of transversely oriented normal surfaces for this application. To facilitate a quick overview, we first state most definitions and results leading up to the main result (Theorem \ref{maintheorem}) as well as the applications, and fill the gaps at the end of the paper. A more general treatment is to be found in \cite{CT}, and references to normal surface theory can be found in \cite{JR}. 

This work was partially supported by NSF grant DMS-0405963. The  authors thank Andrew Casson and Bus Jaco for helpful comments on this project.


\subsection*{Transversely oriented normal surfaces}

\begin{wrapfigure}{l}{28mm}
\begin{center}
  \includegraphics[width=2.5cm]{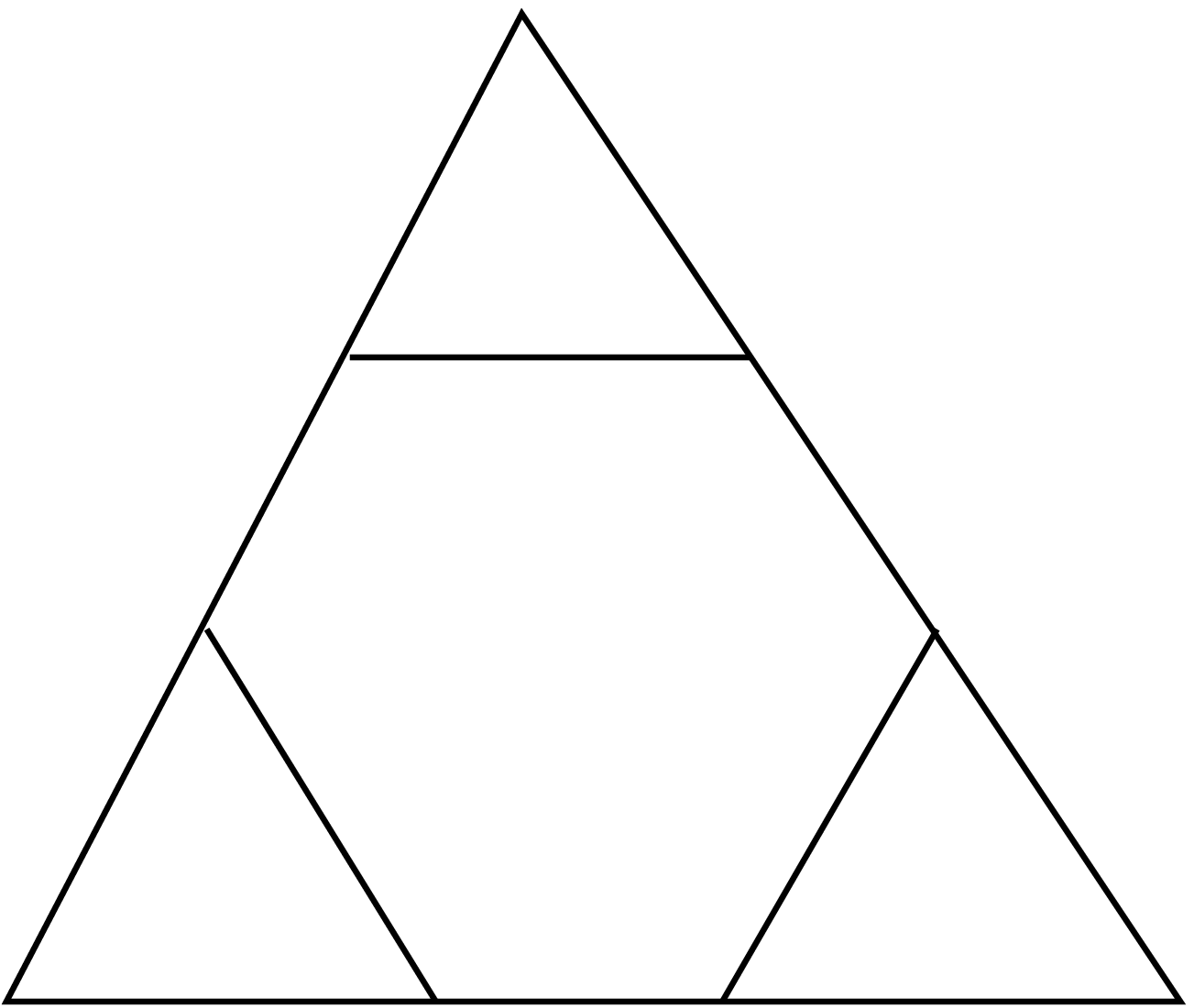}
\end{center}
\end{wrapfigure}
An arc $\alpha$ in a 2--simplex $\Delta$ is a \emph{normal arc} if  $\alpha\cap\partial\Delta=\partial\alpha,$ and in addition the endpoints of $\alpha$ are on distinct edges of $\Delta.$ A \emph{transversely oriented normal arc} $(\alpha,\nu_{\alpha})$ is a normal arc $\alpha$ in $\Delta$ together with a transverse orientation $\nu_{\alpha}$ to $\alpha$ in $\Delta.$ Two such arcs are equivalent if there is a homeomorphism of $\Delta$ to itself which preserves each edge of $\Delta$ and takes one arc with its transverse orientation to the other. There are $6$ equivalence classes.

A disc $D$ in a 3--simplex $\tau$ is a \emph{normal disc} if $D\cap\partial\tau=\partial D$ is a union of normal arcs, no two of which lie in the same face of $\tau.$ The disc $D$ is called a \emph{triangle} if the boundary consists of three normal arcs; otherwise the boundary has four normal arcs and $D$ is 
\begin{wrapfigure}{r}{50mm}
\begin{center}
  \includegraphics[width=4.5cm]{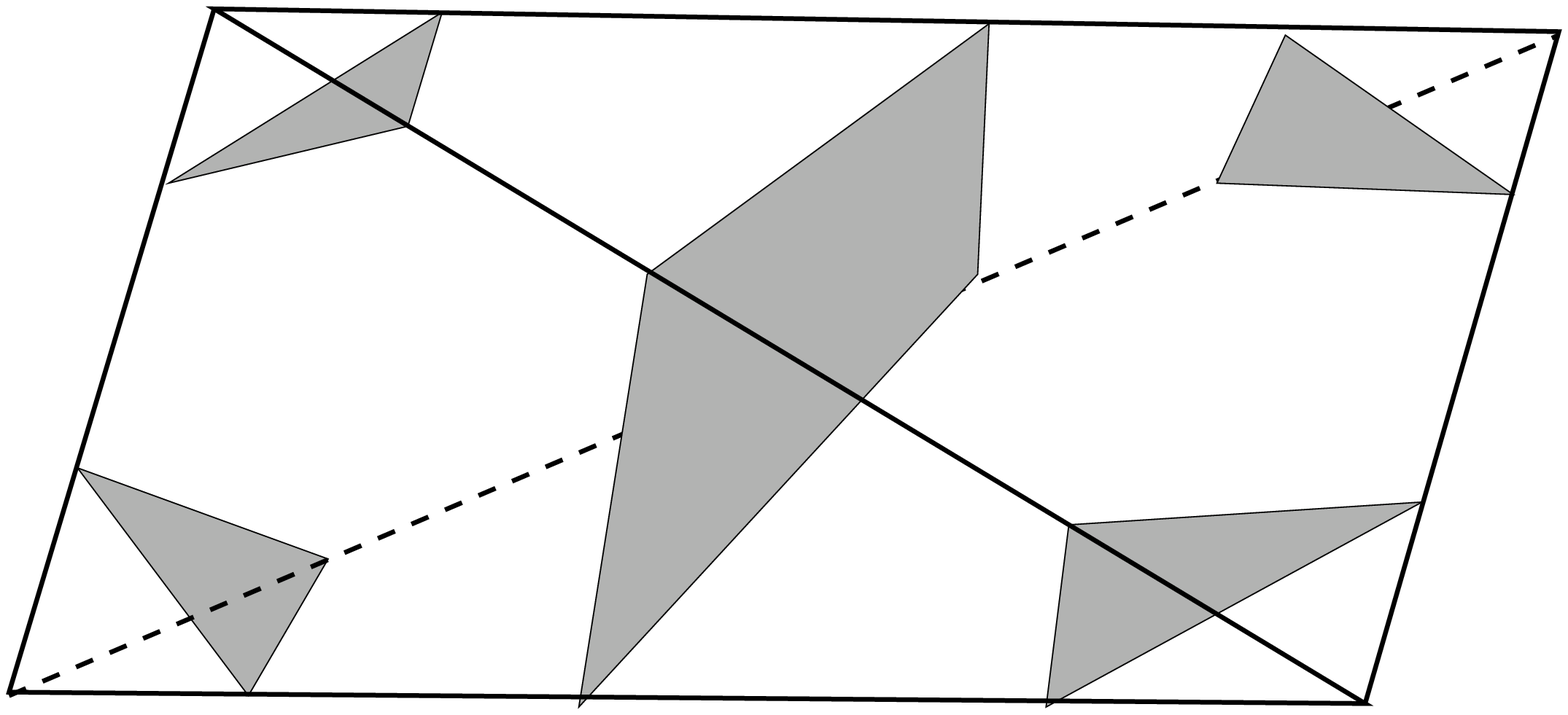}
\end{center}
\end{wrapfigure}
called a \emph{quad.} A \emph{transversely oriented normal disc}  $(D,\nu_D)$ is a normal disc $D$ together with a transverse orientation $\nu_D.$ Two such discs are equivalent if there is a homeomorphism of $\tau$ to itself which preserves each face of $\tau$ and takes takes one disc with its transverse orientation to the other. There are $14$ equivalence classes. The boundary of a transversely oriented normal disc is a collection of transversely oriented normal arcs; three for a triangle and four for a quad. Each  transversely oriented normal arc is contained in the boundary of exactly one equivalence class of transversely oriented normal triangles and one of transversely oriented normal quads in $\tau.$
 
Unless stated otherwise, $M$ denotes a closed, oriented, irreducible 3--manifold and $\tri$ a  triangulation of $M$ with $t$ 3--simplices. It is convenient to allow triangulations of $M$ which are more general than a simplicial triangulation. The interior of every simplex $\sigma$  in $\tri,$ of every dimension, is required to be embedded in $M$, but there may be self-identifications on the boundary of $\sigma.$ A triangulation is \emph{0--efficient} if every normal 2--sphere is \emph{vertex linking}, i.e.\thinspace it bounds a ball contained in a small neighbourhood of a vertex. Jaco and Rubinstein \cite{JR} showed that if $M$ is not homeomorphic to ${\RR}P^3$ or $L(3,1),$ then any minimal triangulation of $M$ is 0--efficient. A \emph{minimal triangulation} of $M$ has the property that no triangulation of $M$ contains fewer 3--simplices.

The \emph{transversely oriented normal disc space}  $ND^{\nu}(\tri)$ is the real vector space of dimension $14t$ with a basis consisting of the equivalence classes of the  transversely oriented normal discs in each 3--simplex of $\tri.$ A properly embedded surface $S\subset M$ is a \emph{normal surface} if the intersection of $S$ with every 3--simplex of $\tri$ is a collection of pairwise disjoint normal discs. If the normal surface $S$ has a transverse orientation, $\nu_S,$ then it determines a unique point $x^{\nu}(S,\nu_S)\in ND^{\nu}(\tri),$ where the coefficient of $(D,\nu_D)$ is the number of transversely oriented normal discs of that type in $S.$ The following result is routine:

\begin{proposition}(Incompressible is isotopic to normal)\label{normalize} 
Every closed, incompressible surface in a closed, irreducible, triangulated $3$-manifold is isotopic to a normal surface.
\end{proposition}

There is a linear subspace $NS^{\nu}(\tri)\subset ND^{\nu}(\tri)$  defined by the \emph{matching equations.} There is one matching equation for each equivalence class of transversely oriented normal arc $(\alpha,\nu_{\alpha})$ in each 2--simplex $\Delta$ of $\tri.$ The two sides of $\Delta$ in $M$ are labelled $+$ and $-$ arbitrarily. The matching equation is
$$t_-+q_-=t_+ + q_+.$$
Here $t_-$ (respectively $q_-$) is the coefficient of the equivalence class of transversely oriented normal triangles (respectively quads) on the $-$ side of $\Delta$ which contains  $(\alpha,\nu_{\alpha})$ in its boundary. Similarly for the $+$ side. This equation expresses that there are the same number of transversely oriented normal discs containing $(\alpha,\nu_{\alpha})$ in their boundary on either side of  $\Delta$. It follows that if $(S,\nu_S)$ is a transversely oriented normal surface in $M,$ then  $x^{\nu}(S,\nu_S)\in NS^{\nu}(\tri).$

\begin{proposition}(Branched immersions)\label{bins}
Every non-zero point in  $NS^{\nu}(\tri)$ with non-negative integral coordinates is represented by a transversely oriented normal branched immersion of a closed orientable surface. Conversely, a transversely oriented normal branched immersion of a closed orientable surface, $(f,\nu_f)\co S\rightarrow M,$ determines a unique point, $x^\nu(f,\nu_f),$ in $NS^{\nu}(\tri).$
\end{proposition}

The \emph{normal surface space} is $NS^{\nu}_+(\tri) = NS^{\nu}(\tri)\cap[0,\infty)^{14t}.$ It is a cone on the compact convex polytope ${\cal C}=NS^{\nu}_+(\tri)\cap V,$ where $V$ is the affine subspace consisting of all points whose coordinates sum to one. Whence ${\cal C}$ is the intersection of a simplex of dimension $14t-1$ with a linear subspace; it is called the \emph{projective solution space} and its extreme points are termed \emph{vertices}. Given a transversely oriented normal surface, we would like to know its Euler characteristic as well as which homology class it represents. This information is given by the linear maps described in the following two results.

\begin{lemma}($\chi^*=\chi$ for immersions)\label{chistar} 
There is a linear map 
$\chi^*\co NS^{\nu}(\tri)\to{\Bbb R}$
with the property that if $(f,\nu_f)\co S\rightarrow M$ is a  transversely oriented normal (unbranched) immersion, then $\chi^*(x^\nu(f,\nu_f))=\chi(S).$
 \end{lemma}
 
The equality in the above lemma does not hold for \emph{branched} immersions. 

If $f\co S\rightarrow M$ is a normal immersion of a closed oriented surface $S,$ then the orientation on $M$ determines an induced transverse orientation for the immersion, $\nu_f.$ Conversely, given a transversely oriented normal branched immersion of a closed surface $S,$ the transverse orientation and the orientation on $M$ determine an orientation of $S.$ This observation leads to the following:

\begin{proposition}(Homology map) \label{homologymap}
There is a surjective homomorphism 
$$h\co NS^{\nu}(\tri)\to H_2(M;{\Bbb R}),$$ 
called the \emph{homology map}, with the following property:
If $(f,\nu_f)\co S\rightarrow M$ is a transversely oriented normal branched immersion, then $f_*([S]) = h(x^{\nu}(f,\nu_f)),$ where $S$ is given the induced orientation.
\end{proposition}

A point $x\in NS^{\nu}_+(\tri)$ is \emph{admissible} if at most one quad type appears in each $3$-simplex (though both orientations are allowed). Two points $x,y\in NS^{\nu}_+(\tri)$ are \emph{compatible} if $x+y$ is admissible. 

\begin{theorem}(Thurston norm via normal surfaces)\label{maintheorem}
Let $M$ be a closed, orientable, irreducible, atoroidal 3--manifold with simplicial or 0--efficient triangulation $\tri.$ Let $\mathcal{C}$ be the projective solution space and $B$ be the convex hull of the finite set of points
$$\left\{\frac{v}{|\chi^*(v)|} : v \text{ is a vertex of } \mathcal{C} \text{ which is admissible and satisfies } \chi^*(v) <  0\right\}.$$
Then $h(B)$ is the unit ball, ${\cal B},$ of the Thurston norm on $H_2(M;{\RR}).$ In particular, ${\cal B}$ has at most $2^{14t}$ vertices.
\end{theorem}

\begin{algorithm}(Fibred faces)\label{algo:fibring}
Let $M$ be a closed, orientable, irreducible, atoroidal 3--manifold. The following is an algorithm to determine the fibred faces of $\mathcal{B}$ and hence to determine whether $M$ fibres over the circle: 

(0) We may assume that $M$ is given via a (simplicial or 0--efficient) triangulation. (1) Compute the unit norm ball of the Thurston norm using the above theorem. (2) Determine the top-dimensional faces. (3) For each top dimensional face $\mathcal{F}$ of $\mathcal{B}:$ (3a) Let $\hat{\mathcal{F}}$ be the barycentre of $\mathcal{F}$ and let $\alpha\in H_2(M;{\Bbb Z})$ be the smallest integral multiple of $\hat{\mathcal{F}}.$  (3b) Find an embedded, norm minimising, transversely oriented normal surface $S$ without sphere or torus components representing $\alpha.$ (3c) Use Haken's algorithm to check whether $\overline{M \setminus S}$ is homeomorphic to $S \times [0,1];$ see \cite{M}. Then $\mathcal{F}$ is a fibred face if and only if Step (3c) yields an affirmative answer. Moreover, $M$ fibres over the circle if and only if at least one face is fibred.
\end{algorithm}

This paper gives a self-contained proof of Theorem \ref{maintheorem} for a 0--efficient triangulation. In the case of a simplicial triangulation, work by Tollefson and Wang \cite{TW} is used to prove Theorem \ref{maintheorem}. An alternative approach to construct ${\cal B}$ is given in \cite{TW}, Algorithm 5.9, and used by Schleimer \cite{S} to obtain an algorithm to determine whether a 3--manifold fibres over the circle. The algorithm to compute ${\cal B}$ given here is practical in the sense that the authors hope to implement it on a computer.

 
 \subsection*{Proofs, neat position and algebraically aspherical solutions}
 
 \begin{proof}[Proof of Proposition \ref{bins}]
A non-zero point in  $NS^{\nu}(\tri)$ with non-negative integral coordinates determines an abstract collection of normal discs. Since the integers satisfy the matching equations, we may identify the edges of these discs in pairs to produce a closed surface $S$ such that there is a map of $S$ into $M$ mapping each disc in $S$ to a normal disc in $M.$ This map is an immersion on the complement of the set, $V,$ of vertices of $S,$ and thus a \emph{branched immersion.} The transverse orientations on the discs match along edges, so this map extends to  a map of $S\times[-1,1]$ into $M$ which is an immersion on the complement of $V\times[-1,1].$ The map of $S$ therefore has a transverse orientation and, since $M$ is orientable, it follows that $S$ is orientable. The converse direction is obvious. 
\end{proof}
 
 \begin{proof}[Proof of Lemma \ref{chistar}]
A \emph{corner} of a $3$-simplex is a small neighborhood of an edge. The \emph{degree} of a 1--simplex $e$ in $\tri$ is the number of corners of $3$-simplices in $\tri$ which contain $e.$ If all these $3$-simplices are embedded this is just the number of $3$-simplices containing $e.$ We define $\chi^*$ on each transversely oriented normal disc $(D,\nu_D)$ as follows. For each corner $v\in\partial D$ let $\delta(v)$ denote the degree of the 1--simplex in $\tri$ that contains $v.$ Then
$$ \chi^*(D,\nu_D) = 1 - \frac{\#(\partial D)}{2}\ +\ \sum_{v\in\partial D}\ \frac{1}{\delta(v)},$$
where $\#(\partial D)$ is the number of sides of $\partial D$ (three for a triangle and four for a quad). A normal immersion of a surface $S$ gives it a cell decomposition into normal triangles and quads. Every 1--cell appears in two 2--cells, and every 0--cell $v$ appears  in $\delta(v)$ 2--cells. Summing the formula then gives the usual formula for the Euler characteristic for $S$ using this cell decomposition.  
\end{proof}
 
\begin{proof}[Proof of Proposition \ref{homologymap}]
Denote by $C_2(M)$ the 2--dimensional singular chain group of $M$ with coefficients ${\Bbb R}.$ A homomorphism $\overline{h}:ND^{\nu}(\tri)\rightarrow C_2(M)$ is first constructed. Suppose $(D,\nu_D)$ is a  transversely oriented normal disc in some $3$-simplex of $\tri.$ The orientation of $M$ and the transverse orientation $\nu_D$ determine an orientation $\mu_D$ of $D$. If $D$ is a normal triangle then $(D,\mu_D)$ may be regarded as a singular 2--simplex. If $D$ is a normal quad, we add a diagonal to subdivide $D$ into two triangles and obtain the sum of two singular $2$-simplices. This defines $\overline{h}$ on a basis and we extend linearly. It is easy to see that if $x\in NS^{\nu}(\tri)$ then $\overline{h}(x)$ is a singular 2--cycle. See \cite{CT} for details. We define $h(x) = [\overline{h}(x)].$ It is also easy to see that if $(f,\nu_f):S\rightarrow M$ is a transversely oriented branched normal immersion then $f_*([S])=h(x^{\nu}(f,\nu_f)).$ The map $h$ is surjective since every homology class is represented by an embedded, closed, oriented surface which may be isotoped to be normal.
\end{proof}

We obtain the theory of normal surfaces if we forget transverse orientations. Thus $ND(\tri)$ is the vector space of dimension $7t$ with basis the equivalence classes of normal discs in each $3$-simplex of $\tri.$ There is one un-oriented matching equation for each equivalence class of normal arc in each $2$-simplex in $\tri.$ The solutions of the matching equations give a linear subspace $NS(\tri)\subset ND(\tri),$ and $NS_+(\tri)=NS(\tri)\cap[0,\infty)^{7t}.$  A point in $NS_+(\tri)$ is \emph{admissible} if in each $3$-simplex the coefficient of at most one quad type is non-zero.    
 
We now the describe the \emph{geometric sum} of two embedded normal surfaces. Suppose $D,D'$ are two normal discs in a $3$-simplex $\tau$ which intersect transversely in an arc $\alpha.$ After cutting $D$ and $D'$ along $\alpha$  there are two possible ways to cross-join to obtain two disjoint discs. There is  exactly one way which yields disjoint normal discs unless $D$ and $D'$ are quads of different types. In that case there is no cross join which gives disjoint normal discs. The direction of the cross join is determined by the way the normal arcs intersect in the boundary. Now suppose $S$ and $F$ are two normal surfaces. We may isotope them so they intersect transversely and, for every 3--simplex $\tau\in \tri,$ the intersection of a normal disc in $S\cap \tau$ with one in $F\cap\tau$ is either empty or a single arc.  If $S$ and $F$ are \emph{compatible}, i.e.\thinspace the sum of their normal coordinates is admissible, then there is a unique way to cut and cross join $S$ and $F$ along $S\cap F$ to obtain a new (possibly not connected) normal surface denoted $S+F.$

 The pair $(S,F)$ is \emph{disc reduced} if no curve of intersection in $S \cap F$ bounds a disc in both $S$ and $F.$ An innermost discs argument shows that $S$ and $F$ can be replaced by normal surfaces $S'$ and $F'$ homeomorphic to (but possibly not normally isotopic to) $S$ and $F$ respectively such that $(S',F')$ is disc reduced and $S'+F'$ is normally isotopic to $S+F.$ An ambient isotopy of $M$ is a \emph{normal isotopy} if it preserves every cell in the triangulation $\tri.$ The following is well known: 

\begin{proposition}(Admissible implies unique embedded normal surface)\label{unoriented} 
Suppose $M$ is a closed, triangulated $3$-manifold. There is a one--to--one correspondence between admissible non-negative integral points  in $NS(\tri)$ and normal isotopy classes of \emph{embedded} normal surfaces in $M.$ Furthermore, addition of compatible points corresponds to geometric sum of compatible normal surfaces.
\end{proposition}
\begin{proof} 
Assume for simplicity that every simplex of $\tri$ is embedded. The general case only involves more words. A point $x\in NS_+(\tri)$ determines a collection of normal discs in each $3$-simplex of $\tri.$ These in turn determine a collection of points on each $1$-simplex of $\tri.$ These points determine a collection of normal arcs in each $2$-simplex of $\tri.$ Up to normal isotopy there is a unique way to embed these arcs in each $2$-simplex. Thus up to normal isotopy $x$ determines a unique set of disjoint simple closed curves, $C(x,\tau),$ in the boundary of each $3$-simplex $\tau\in\tri.$ Up to normal isotopy there is a unique set of disjoint discs in $\tau$ with boundary $C(x,\tau).$ These discs are normal discs if and only if each simple closed curve in $C(x,\tau)$ meets each face of $\tau$ at most once. This happens if and only if $x$ is admissible. If $x$ is admissible these normal discs are identified along their boundaries to produce a normal surface $S$ which is thus unique up to normal isotopy. 
\end{proof}

The corresponding result in the transversely oriented setting is weaker:

\begin{proposition}(Admissible implies immersed transversely oriented normal)\label{admimm} 
If $M$ is a closed, orientable, triangulated 3--manifold, then every admissible non-negative integral point in $NS^{\nu}(\tri)$ is the coordinate of a transversely oriented (unbranched) normal  immersion of a closed oriented surface.
\end{proposition}
\textbf{Proof  } A non-negative admissible integral point $x^{\nu}\in NS^{\nu}(\tri)$ determines a collection of transversely oriented normal discs in each $3$-simplex. We embed each of these discs in the simplex in a canonical \emph{neat position}.  The boundary arcs of different discs may intersect but only in the interior of $2$-simplices in $\tri,$ and \emph{canonical} means that the normal discs in adjacent $3$-simplices have boundary arcs that match pairwise with matching transverse orientations to give a transversely oriented immersed normal surface realising $x^{\nu}.$

First we describe \emph{neat position.}  Suppose $(\alpha,\nu_{\alpha})$ is a transversely oriented normal arc in a $2$-simplex $\Delta.$  The transverse orientation is thought of as a function on the  
\begin{wrapfigure}{r}{37mm}
\begin{center}
  \includegraphics[width=35mm]{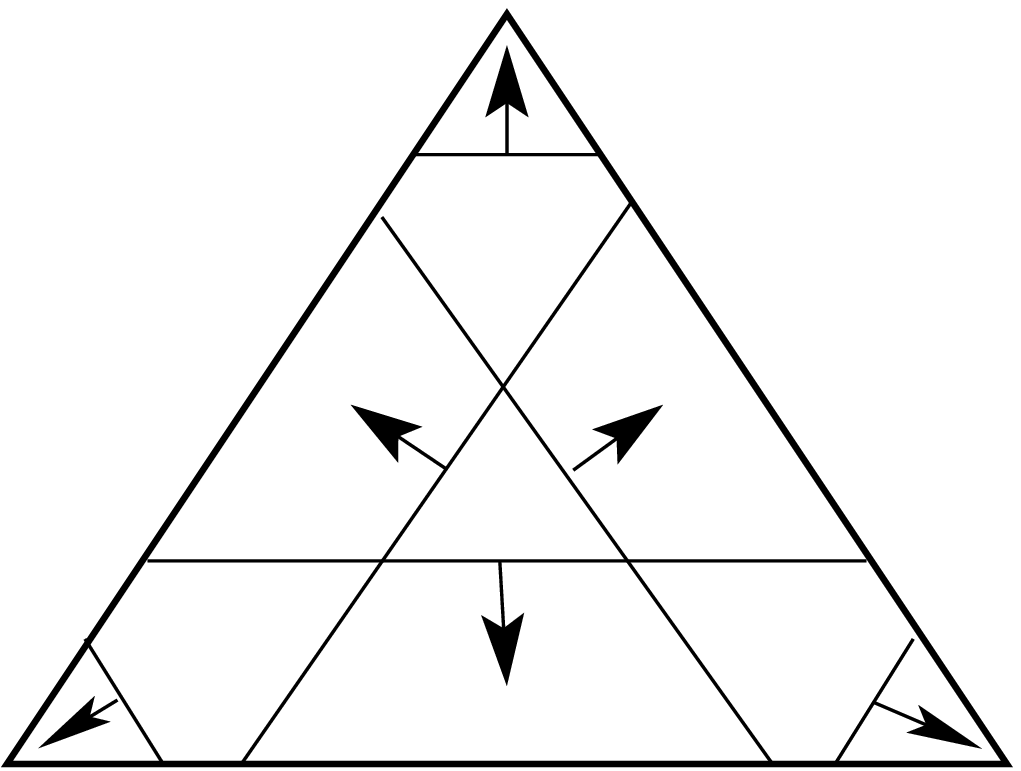}
\end{center}
\end{wrapfigure}
components of $\Delta\setminus \alpha$ which sends one component to $+1$ and the other to $-1.$ Denote the former by $C_+$ and let $\sigma$ denote the maximal sub-simplex of $\Delta$ contained in $C_+.$ There are $6$ such simplices and they are in bijective correspondence with the equivalence classes of $(\alpha,\nu_{\alpha}).$ The arc is called \emph{short} if $\sigma$ is a 0--simplex. Otherwise $\sigma$ is a 1--simplex and the arc is called \emph{long.} A family of transversely oriented arcs is in \emph{neat position} if each arc is straight, two arcs intersect only in the interior of $\Delta$, and the only arcs which intersect are long arcs of different types.

A family of transversely oriented normal discs in a $3$-simplex $\tau$ is in \emph{neat position} if the intersection of the boundary of the family with each $2$-simplex $\Delta\subset\partial\tau$ is in neat position and each normal disc $D$ is the cone on $\partial D$ from the barycentre of the vertices of $\partial D.$ Thus $D$ is determined by the intersection points of $\partial D$ with the $1$-skeleton $\tri^1.$ Hence in neat position every triangle is flat. A transversely oriented normal triangle is \emph{small} if the boundary consists of three short arcs and otherwise the boundary consists of three long arcs and the triangle is called \emph{large.} It follows that in neat position, two triangles intersect if and only if they are both large and of different un-oriented types.

Suppose $(D,\nu_D)$ is a transversely oriented normal disc in a $3$-simplex $\tau.$ As above, a component of $\tau\setminus D$ is labelled $C_+$ using $\nu_D.$ Let $\sigma(D,\nu_D)$ denote the maximal sub-simplex of $\tau$ contained in $C_+.$ There are $14$ such simplices and they are in 1-1 correspondence with the equivalence classes of transversely oriented normal discs. If a family of discs is in neat position, then, after an ambient isotopy, every disc $(D,\nu_D)$ in the family is in a small neighborhood of $\sigma(D,\nu_D).$ 

A triangle is  \emph{small} if $\sigma$ is a $0$-simplex and \emph{large} if $\sigma$ is a 2--simplex. For a quad $\sigma$ is always a 1--simplex and $D$ will be chosen to be a long thin rectangle very close to $\sigma.$ In neat position, if two discs intersect then they are either two large triangles of different types, or a quad $(D,\nu_D)$ and a large triangle which meets $\sigma(D,\nu_D).$ 

Suppose we are given a collection of transversely oriented normal discs in a 3--simplex $\tau$ which is admissible: only one quad type appears but both transverse orientations are allowed. We claim that up to normal isotopy there is a unique way to position these discs in neat position in $\tau.$

For existence, it suffices to show that it is possible to place in neat position one copy of each transversely oriented normal triangle in $\tau$ together with two normal quads
\begin{wrapfigure}{l}{55mm}
\begin{center}
  \includegraphics[width=5cm]{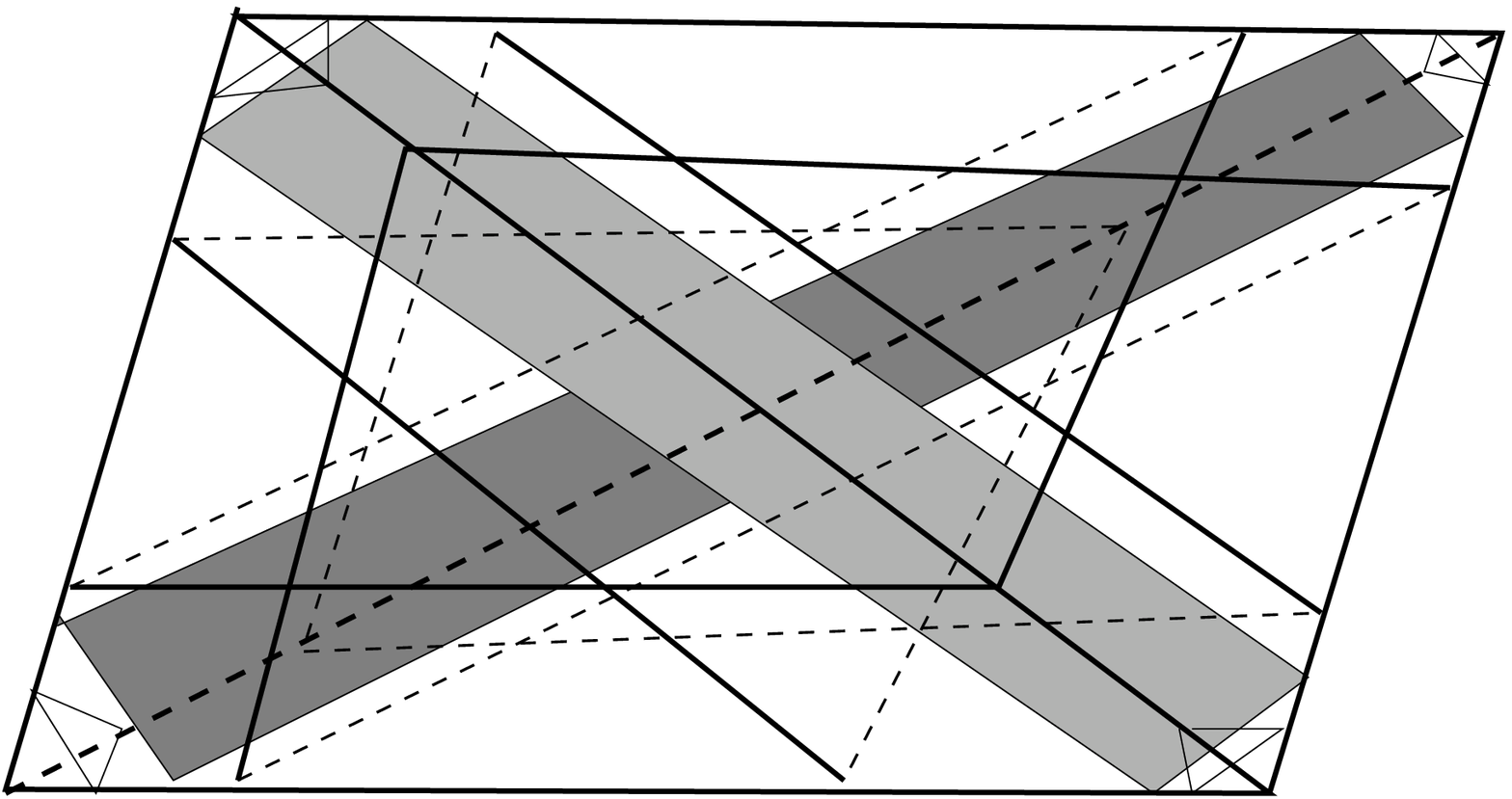}
\end{center}
\end{wrapfigure}
which differ only in transverse orientation. For, having done this, one may take closely spaced parallel copies of these discs. The figure shows how to place the discs. First place the large triangles in neat position. Next place each transversely oriented quad in a small neighborhood of the 1--simplex of $\sigma$ it determines, so that it is disjoint from the large triangles it can be made disjoint from. Finally place the small triangles in small neighbourhoods of the vertices disjoint form all other discs. This proves existence.

Uniqueness follows from the fact that the family of neat arcs in each $2$-simplex $\Delta$ of $\sigma$ is determined up to normal isotopy by their endpoints, and the fact that neat discs are determined by their boundaries.  This completes the claim.

We now describe \emph{canonical} neat position for admissible solutions.  A vector $x^{\nu}$ in $NS_+(\tri)$ determines, for each 1--simplex $e$ of $\tri,$ a finite number of transversely oriented points on $e.$ We first position these points on $e$ in \emph{neat position} which just means  close to the end of $e$ determined by the transverse orientation on that point, and equally spaced apart.  The vector $x^{\nu}$ also determines a family of transversely oriented normal arcs types in each $2$-simplex of $\tri.$  By the previous discussion, there is a unique way to place these normal arcs in neat position with endpoints the chosen neatly positioned points.
Since $x^{\nu}$ is admissible, there is then a unique way to place the transversely oriented normal disc types determined by $x^{\nu}$ in each 3--simplex of $\tri$ into neat position so the boundary is the set of neatly positioned arcs. This is the canonical neat position. 

Since $M$ is orientable and the immersed surface constructed above is transversely orientable it is also orientable.\qed

An orientable surface $S$ is \emph{aspherical} if no connected component of $S$ is a sphere. We wish to have an algebraic criterion for deciding whether a transversely oriented normal surface is aspherical depending only on its transversely oriented normal coordinates.
There is a partial order $\ge$ on $ND^{\nu}(\tri),$ where $x\ge y$ if and only if every coordinate of $x$ is greater than or equal to the corresponding coordinate of $y.$

\begin{definition*}
A point $x \in NS^{\nu}_+(\tri)$ is \emph{algebraically aspherical} if for all $y \in NS^{\nu}_+(\tri)$ such that $x\ge y,$ we have $\chi^*(y) \le 0.$
\end{definition*}

\begin{lemma}(Properties of aspherical elements)\label{asphs}
The following properties follow directly from the definition.
\begin{enumerate}
\item\label{properties:non-pos} If $x \in NS^{\nu}_+(\tri)$ is algebraically aspherical, $\chi^*(x) \le 0.$
\item\label{properties:sum of alg irr} If $x= \sum x_i$ is algebraically aspherical and $x_i \in NS^{\nu}_+(\tri)$ for each $i,$ then each $x_i$ is algebraically aspherical.
\item\label{properties:multiplicative} If $x \in NS^{\nu}_+(\tri)$ is algebraically aspherical, so is $\alpha x$ for any $\alpha > 0.$
\end{enumerate}
\end{lemma}

\begin{proposition}(embedded aspherical in 0--efficient $\Rightarrow$ algebraically aspherical)\label{admissibleaspherical} 
Let $M$ be a closed, oriented 3--manifold with 0--efficient triangulation $\tri.$ If $(F, \nu_F)$ is an embedded transversely oriented normal surface in $M$ and no component of $F$ is a sphere, then $x^\nu(F, \nu_F)$ is admissible and algebraically aspherical.
\end{proposition}
\begin{proof}
Since $F$ is embedded $x^\nu=x^\nu(F, \nu_F)$ is admissible. Forgetting transverse orientations gives a linear map $\varphi \co NS^\nu(\tri) \to NS(\tri),$ and we write $x=\varphi(x^{\nu}).$  The topology of the embedded normal surface $F$ is determined (up to normal isotopy) by $x.$ If $x^\nu$ is not algebraically aspherical, we will show that there is an integer $n>0$ such that $n x>y,$ where $y=x(S)$ for some embedded normal sphere $S.$ Then
$$n x = (n x- y) + y,$$
where $nx, (nx-y), y\in NS_+(\tri)$ are compatible admissible non-negative integral solutions to the unoriented matching equations. By Proposition \ref{unoriented}  for each admissible $z\in NS_+(\tri)$ there is an embedded normal surface $F(z)$ in $M$ with normal coordinates $z,$ and this surface is unique up to normal isotopy. Thus $F(y)$ is normally isotopic to $S.$ Since  $\tri$ is 0--efficient $S$ is a vertex linking sphere and so may be isotoped into a small neighborhood of that vertex disjoint from $F(nx-y).$  Let $n\cdot F$ denote an embedded normal surface consisting of $n$ parallel disjoint copies of the normal surface $F.$ The uniqueness part of Proposition \ref{unoriented} implies $F(nx) = n\cdot F.$  The above equation implies $n\cdot F$ is normal-isotopic to the geometric sum of $F(nx-x(S))$ and $S.$ Since these surfaces are disjoint, their geometric sum equals their disjoint union. Thus $n\cdot F$ contains a sphere normally isotopic to $S,$ which is a contradiction. 

It remains to prove that if $x^{\nu}$ is not algebraically aspherical, then $S$ exists. If $x^{\nu}$ is not algebraically  aspherical, then there is $a^{\nu}\in NS^{\nu}_+(\tri)$ with $x^{\nu}\ge a^{\nu}$ and $\chi^*(a^{\nu})>0.$ We may assume $a^{\nu}$ has non-negative rational coordinates. Then there is an integer $n>0$ so that $b^{\nu}=na^{\nu}$ has integer coordinates and $nx^{\nu}\ge b^{\nu}.$ Now $\chi^*(b^{\nu})=n\chi^*(a^{\nu})>0.$ Since $x^{\nu}$ is admissible, $b^{\nu}$ is admissible and hence $b=\varphi(b^{\nu})\in NS_+(\tri)$ is admissible. The map $\chi^*$ factors through a linear map $\chi_{un}^*:ND(\tri)\rightarrow {\Bbb R}$ with the property that if $A$ is an embedded normal surface, then $\chi_{un}^*(x(A))=\chi(A).$ So $\chi^*(b^{\nu})=\chi^*_{un}(\varphi(b^\nu))$ and since $F(b)$ is an embedded normal surface, $\chi(F(b))=\chi_{un}^*(b)=\chi^*(b^{\nu}).$ Hence $\chi(F(b))>0.$ Thus $F(b)$ contains a component $S$ with $\chi(S)>0.$ Since $S\subset F(b)$ are normal surfaces it follows that $b\ge x(S).$ If $S={\Bbb R}P^2$ then, since $M$ is orientable, a small regular neighborhood of $S$ is bounded by a 2--sphere $S',$ and $x(S')=2x(S).$ Since $\tri$ is 0--efficient, $S'$ is a vertex linking sphere. But this implies that $x(S)$ is not an integral solution, giving a contradiction. Thus $S$ is a sphere and $nx\ge b\ge x(S).$
 \end{proof}

If $S$ is a closed surface, then $-\chi_-(S)$ is the sum of the Euler characteristic over all connected components of $S$ with negative Euler characteristic (see \cite{Thu1986}). If $c\in H_2(M),$ then the \emph{Thurston norm} of $c$ is
$$||c||\ =\ \inf_{S: [S]=c}\ \chi_-(S).$$
If $M$ is irreducible, discarding sphere components from  $S$ does not change $[S]$ or $\chi_-(S)$ and so the infimum can be taken over surfaces containing no spheres.

A transversely oriented normal surface $(F, \nu_F)$ in $M$ is called a \emph{taut normal surface} if it is aspherical and $\chi(F) = - || h(x^\nu(F, \nu_F)) ||,$ and it is \emph{least weight taut} if $F$ minimises $|F \cap \tri^{(1)}|$ over all taut normal surfaces representing $h(x^\nu(F, \nu_F)).$ The following is a re-formulation of (part of)  Lemma 3.2 of Tollefson and Wang \cite{TW}.

\begin{lemma}[Tollefson--Wang \cite{TW}]\label{lem:TW}
Let $M$ be a closed, oriented 3--manifold with simplicial triangulation $\tri.$ Let $(F, \nu_F)$ be least weight taut. Let $G,$ $H$ be orientable normal surfaces such that $m x(F) = x(G) + x(H)$ for some positive integer $m$ and the pair $(G,H)$ is disc reduced. Then the transverse orientation of $F$ induces unique transverse orientations $\nu_G$ and $\nu_H$ on $G$ and $H$ respectively  such that $x^\nu(F, \nu_F) = x^\nu(G, \nu_G) + x^\nu(H, \nu_H)$ and $(G, \nu_G)$ and $(H, \nu_H)$ are least weight taut.
\end{lemma}

\begin{proposition}(least weight taut $\Rightarrow$ algebraically aspherical)\label{lw-taut=>aspherical} 
Let $M$ be a closed, oriented, irreducible 3--manifold with simplicial triangulation $\tri.$ If $(F, \nu_F)$ is a least weight taut normal surface in $M,$ then $x^\nu(F, \nu_F)$ is admissible and algebraically aspherical.
\end{proposition}

\begin{proof}
The proof of the previous proposition is modified as follows. First note that if $(F(n x- y),F(y))$ is not disc reduced, then there is $y'$ such that $(F(n x- y'),F(y'))$ is disc reduced and $\chi(F(y'))=\chi(F(y))=2.$ Lemma \ref{lem:TW} implies that $F(y')$ can be transversely oriented to give a least weight taut normal surface. Since $F(y')\neq \emptyset,$ it follows that a sphere represents a non-trivial homology class contradicting the assumption that $M$ is irreducible.
\end{proof}

\begin{proof}[Proof of Theorem \ref{maintheorem}] 
We claim that $||h(y)||\le|\chi^*(y)|$ for all $y\in \cone(B),$ the cone of $B$ over the origin. It then follows that $h(B)\subseteq{\cal B}$ since $B = \cone(B)\cap\chi^*(-1).$ First we prove the claim in the special case when $y=v$ is a vertex of $B.$ Since the inequalities defining ${\cal C}$ have integral coefficients, $v$ has rational coordinates. Let $t>0$ be minimal subject to $x=t\cdot v$ is an integral point. Since $B\subset NS^{\nu}_+(\tri),$ $x$ is non-negative. Since $v$ is admissible, $x$ is admissible. By Proposition \ref{admimm}  there is a transversely oriented (unbranched) normal immersion $(f,\nu_f)\co S\rightarrow M$ such that $x^{\nu}(f,\nu_f)=x,$ and Proposition~\ref{homologymap} gives $h(x)=f_*([S])\in H_2(M).$  Since $S$ is (unbranched) immersed, Lemma~\ref{chistar} yields $\chi^*(x)=\chi(S).$   If $S$ is not connected, then $x$ is the sum over all components $S_i\subset S$ of $x(S_i)\in \cone({\cal C}).$ Since $v$ is an extreme point of ${\cal C}$ it is not a non-trivial convex combination of points in ${\cal C}.$ However $x^{\nu}(f,\nu_s) = \sum_i x^{\nu}(f|S_i,\nu_f|S_i).$ Thus $x^{\nu}(f,\nu_s)$ is a multiple of $x^{\nu}(f|S_i,\nu_f|S_i)$ for every component $S_i$ of $S.$ By minimality of $t$ it follows that $S$ is connected. Now $\chi^*(x)=t\chi^*(v)<0$ thus $S$ is not a sphere, hence $-\chi_-(S)=\chi(S)=\chi^*(x).$  Gabai \cite{g} showed that the embedded norm equals the singular norm, and so $||h(x)|| \le \chi_-(S)= |\chi(S)|=|\chi^*(x)|.$ This proves the claim in the special case. 

To prove the claim in the general case, assume $y=\sum \lambda_iv_i$ with $\lambda_i\ge0,$ where each $v_i$ is a vertex of $B.$ Then $h(y)=\sum \lambda_ih(v_i)$ so $||h(y)||\le\sum \lambda_i||h(v_i)||. $ Now $\chi^*(y)=\chi^*(\sum \lambda_i v_i) =\sum\lambda_i\chi^*(v_i).$ Since $\chi^*(v_i)< 0$ we get  $|\chi^*(y)| =\sum\lambda_i|\chi^*(v_i)|.$ By the special case, $||h(v_i)||\le|\chi^*(v_i)|,$ and it follows that $||h(y)||\le|\chi^*(y)|,$ proving the claim.

To prove the reverse containment, a rational point in $\beta\in\partial{\cal B}$ can be expressed as $\beta=[S]/|\chi(S)|$ for some norm-minimising, transversely oriented embedded surface $(S,\nu_S)$ no component of which is a sphere or a torus. By Proposition~\ref{normalize} $S$ can be isotoped into normal position. The ray through $p=x^{\nu}(S,\nu_S)$ intersects ${\cal C}$ at a point $x.$ Since $p$ is admissible, $x$ is admissible. Lemma~\ref{chistar} gives $|\chi^*(p)| = |\chi(S)|= ||[S]||.$ Since $S$ is embedded and aspherical, Proposition \ref{admissibleaspherical} implies that $p$ is algebraically aspherical if $\tri$ is 0--efficient. Otherwise, it follows from \cite{TW}, Lemma 2.1, that $(S,\nu_S)$ may be chosen to be least weight taut in which case Proposition \ref{lw-taut=>aspherical} implies that $p$ is algebraically aspherical. Hence $x$ is algebraically aspherical by Lemma~\ref{asphs}(\ref{properties:multiplicative}). We can express $x$ as a convex linear combination of some of the vertices of ${\cal C}.$ Then $x=\sum t_i\cdot x_i,$ where each $x_i$ is a vertex of ${\cal C}$ and $0<t_i\le1.$  Parts (\ref{properties:sum of alg irr}) and (\ref{properties:multiplicative}) of  Lemma \ref{asphs} imply that each $x_i$  is algebraically aspherical, and therefore by Lemma~\ref{asphs}(1) $\chi^*(x_i)\le0.$ Moreover, $x_i\le x$ therefore $x_i$ is admissible. If $\chi^*(x_i)=0,$ then the smallest integral multiple of $x_i$ is the normal coordinate of a connected, immersed surface of zero Euler characteristic, giving $h(x_i)=0$ since $M$ is atoroidal. 

Let $\hat{x}= \sum t_i\cdot \hat{x}_i,$ where $\hat{x}_i=x_i$ if $\chi^*(x_i)<0$ and $\hat{x}_i=0$ otherwise. Hence $h(\hat{x})=h(x),$ $\chi^*(x)=\chi^*(\hat{x})$ and each non-zero $\hat{x}_i$ is in $\cone(B).$ Since $\cone(B)$ is convex, $\hat{x} \in \cone(B).$ It follows that $\hat{x}/|\chi^*(\hat{x})|\in \cone(B)\cap\chi^*(-1)=B.$ Now $[S]=h(p)$ and 
$$h(\hat{x})/|\chi^*(\hat{x})|=h(x)/|\chi^*(x)| =h(p)/|\chi^*(p)|=[S]/|\chi(S)|=\beta\in\partial{\cal B}.$$
Hence $h(B)$ contains all the rational points in $\partial{\cal B}.$ The set of rational points in $\partial{\cal B}$ is dense because the Thurston norm takes integral values on integral points.  Thus $\partial\mathcal{B}\subseteq h(B).$ Since $h$ is linear and $B$ is convex,  $\mathcal{B} \subseteq h(B).$

Each vertex of ${\cal C}=\Delta^{14t-1}\cap NS^{\nu}(\tri)$ lies in a unique minimal sub-simplex of $\Delta.$ The sub-simplices of $\Delta^{14t-1}=V\cap[0,\infty)^{14t}$ correspond to coordinate subspaces of ${\Bbb R}^{14t}.$   Thus a vertex of $\partial {\cal C}$ is uniquely determined by which normal coordinates are zero. There are $14t$ coordinates, so ${\cal C}$ has at most $2^{14t}$ vertices. Every vertex of the unit ball of the Thurston norm ball, ${\cal B},$ is the image of a point on a ray through a vertex of ${\cal C}.$ This gives the claimed bound on the number of vertices of ${\cal B}.$ 
\end{proof}

\begin{proof}[Proof of Algorithm \ref{algo:fibring}]
Only part (3b) needs to be explained. Let $B_\le$ be the convex hull of the finite set of all points $\frac{v}{|\chi^*(v)|},$ where $v$ is a vertex of $\mathcal{C}$ which is admissible and satisfies $\chi^*(v) \le  0.$ Define
$$Q = \cone(B_\le) \cap \{x :  h(x)= \alpha\} \cap \{ x : \chi^*(x) = -|| \alpha || \}.$$
The \emph{total weight} of a point in $Q$ is the sum of its coordinates. Given an integer $w>0$ let $\mathcal{P}(w)$ be the finite (possibly empty) set consisting of all admissible integral points in $Q$ of total weight at most $w.$ The algorithm to construct $S$ is to increase $w$ until one of the points in $\mathcal{P}(w)$ is found, by the following procedure, to be the coordinate of an embedded transversely oriented normal surface without sphere or torus components. Given $x^\nu \in\mathcal{P}(w),$ construct the unique embedded normal surface $F$ with un-oriented normal coordinate $x=\varphi(x^\nu).$ Discard $x^\nu$ if some component of $F$ is a torus, a sphere or 1--sided. Otherwise check whether the components of $F$ can be transversely oriented to yield an embedded transversely oriented normal surface with coordinate $x.$ 

We now prove this algorithm terminates. Let $F$ be a norm-minimising oriented surface without sphere or torus components and with $[F]=\alpha\in H_2(M),$ so $|| \alpha || =-\chi(F).$ By Proposition \ref{normalize} we may assume $F$ is a transversely oriented normal surface and set $p=x^{\nu}(F,\nu_F).$ Then $p\in Q$ and $p\in\mathcal{P}(w)$ for $w$ sufficiently large. Hence the algorithm will construct $F$ or a surface with the same properties of lower weight.
\end{proof}

\begin{remark}
If $M$ is a closed, oriented 3--manifold, then $h(B_\le)$ is the unit ball of the Thurston norm --- this is a non-compact polytope if $M$ is not atoroidal.
\end{remark}



\address{Department of Mathematics, University of California Santa Barbara, CA 93106, USA}
\email{cooper@math.ucsb.edu}

\address{Department of Mathematics and Statistics, The University of Melbourne, VIC 3010, Australia} 
\email{tillmann@ms.unimelb.edu.au} 
\Addresses

\end{document}